\documentclass[11pt,a4paper]{amsart}
\usepackage{amsmath, amsthm, amssymb, amsfonts}
\usepackage{mathrsfs}

\newtheorem{thm}{Theorem}

\newtheorem{lem}[thm]{Lemma}

\newtheorem*{thm*}{Theorem}
\newtheorem*{cor*}{Corollary}
\newtheorem*{lem*}{Lemma}
\newtheorem*{prop*}{Proposition}

\theoremstyle{definition}

\newtheorem*{defn*}{Definition}

\newtheorem*{prob*}{Problem}

\theoremstyle{definition}

\theoremstyle{definition}

\newtheorem*{conj*}{Conjecture}

\newtheorem*{example*}{Example}

\renewcommand{\today}{\number \day \space \ifcase \month \or January\or%
  February\or March\or April\or May\or June\or July\or August\or%
  September\or October\or November\or December\fi \space \number \year}

\usepackage{mathptmx}
\begin{document}
\title{Asymptotic Enumeration of Labelled Interval Orders}
\author{Graham Brightwell}
\email{G.R.Brightwell@lse.ac.uk}
\author{Mitchel T.\ Keller}
\thanks{The second author was supported in this research by a Marshall
  Sherfield Fellowship.}
\email{M.T.Keller@lse.ac.uk}
\address{Department of Mathematics\\London School of Economics\\Houghton Street\\London\\WC2A 2AE\\United Kingdom}
\date{29 November 2011}
\maketitle

\begin{abstract}
Building on work by Zagier, Bousquet-M\'elou et al., and Khamis,
we give an asymptotic formula for the number of labelled interval orders on an $n$-element set.
\end{abstract}

\section{Introduction}

In his monograph \cite{fishburn:intordsbook}, Fishburn notes the
absence of definitive results for the enumeration of interval orders
and labelled semiorders, standing in contrast to the situation for
interval graphs. In \cite{hanlon:intgraph-enum}, Hanlon provides
extensive enumerative results for various classes of interval
graphs. However, until recently, the question of interval
\emph{orders} was untouched. In \cite{bousquet:intord-enum},
Bousquet-M\'elou et al. established a relationship between a
particular class of sequences of nonnegative integers they termed
\emph{ascent sequences} and unlabelled interval orders. This allowed
them to determine the ordinary generating function for the number
unlabelled interval orders. For additional results building on the
work of Bousquet-M\'elou on ascent sequences, see
\cite{claesson:matchings-posets, 
  dukes:intord-indist, dukes:ascent-upper-triangular, kitaev:intord-stats}.

Bousquet-M\'elou et al. also established a connection to Vassiliev invariants of knots via Stoimenow's relation of them to what he termed \emph{regular linearized chord diagrams} in \cite{stoimenow:chord-diagr}. They combined this with Zagier's work in \cite{zagier:chord-enum}, which determined a generating function for the number of regular linearized chord diagrams and derived an asymptotic formula for its coefficients. The combination of these results addressed half of Fishburn's inquiry regarding enumerating interval orders as
\begin{thm}[Zagier \cite{zagier:chord-enum}; Bousquet-M\'elou et al. \cite{bousquet:intord-enum}]\label{thm:unlabelled}
  The generating function for the number of unlabelled interval orders with $n$ points is
  \[I(x) = \sum_{n\geq 0}i_n x^n = \sum_{n\geq 0}\prod_{i=1}^n (1-(1-x)^i).\]
  Its coefficients satisfy \[i_n \sim n!\sqrt{n}\left(\frac{6}{\pi^2}\right)^n \left(C_0 +\frac{C_1}{n} + \frac{C_2}{n^2} + \cdots\right),\]
  with
  \[C_0 = \frac{12\sqrt{3}}{\pi^{5/2}}e^{\pi^2/12}\]
  and the remaining $C_i$ explicitly computable.
\end{thm}

A poset is called \emph{rigid} if its automorphism group is trivial.  For an interval order $P=(X,<)$, the family of sets
$U(x) = \{ y\in X : y > x\}$, for $x \in X$, is totally ordered by inclusion.  Similarly, the family of sets $D(x) = \{ y \in X : y < x\}$, for
$x \in X$, is totally ordered by inclusion.  If the interval order $P$ has a non-trivial automorphism, say taking $x$ to $y$, then $|D(x)| = |D(y)|$
and $|U(x)| = |U(y)|$, and therefore $D(x) = D(y)$ and $U(x) = U(y)$.  In other words, for interval orders,
being rigid is equivalent to what Trotter \cite{trotter:dimbook} terms having ``no duplicated holdings''; in other words, no pair of elements
$\{x,y\}$ with $D(x) = D(y)$ and $U(x) = U(y)$. Building on the work of Bousquet-M\'elou et al., Khamis obtained the generating function for the number of rigid unlabelled interval orders in \cite{khamis:intord-enum}. His result is

\begin{thm}[Khamis \cite{khamis:intord-enum}]\label{thm:rigid}
  The generating function for the number of unlabelled rigid interval orders with $n$ points is
  \[R(x) = \sum_{n\geq 0}r_nx^n = \sum_{n\geq 0}\prod_{i=1}^n\left(1-\frac{1}{(1+x)^i}\right).\]
\end{thm}

It is useful to note the following relationships between the generating functions $I(x)$ and $R(x)$:
\[I(x) = R\left(\frac{x}{1-x}\right)\qquad\text{and}\qquad R(x) = I\left(\frac{x}{1+x}\right). \]
Khamis observes \cite[Theorem 5.1]{khamis:intord-enum} that $r_n = 2^{n\lg(n) + o(n\lg n)}$, but does not obtain any more precise results
about the rate of growth of $r_n$. In this note, we obtain a precise asymptotic formula for $r_n$, and use this to obtain an asymptotic formula
for the number $\ell_n$ of \emph{labelled} interval orders on $n$ points.

Our main result is that
\[ \ell_n = \frac{12 \sqrt{3}}{ \pi^{5/2}} (n!)^2 \sqrt{n} \left(\frac{6}{\pi^2}\right)^n \left( 1 + O(1/n) \right). \]

\section{Asymptotic Enumeration of Unlabelled Rigid Interval Orders}

In this section, we use the generating function found by Khamis~\cite{khamis:intord-enum} to establish an asymptotic formula for the number
$r_n$ of unlabelled rigid interval orders. In particular, we shall prove the following theorem.

\begin{thm}\label{thm:rigid-asympt}
  The number of unlabelled rigid interval orders on $n$ points is
\[ r_n \sim n! \sqrt n \left( \frac{6}{\pi^2} \right)^n \left( D_0 + \frac{D_1}{n} + \frac{D_1^2}{n} + \cdots \right) \]
with
\[ D_0 = \frac{12\sqrt{3}}{\pi^{5/2}e^{\pi^2/12}}, \]
and further $D_i$ explicitly computable.
\end{thm}



\begin{proof}
To establish the asymptotic formula, we begin by recalling that
  \[R(x) = I\left(\frac{x}{1+x}\right).\]
Therefore, $r_n$ is the coefficient of $x^n$ in
  \[\sum_{j=0}^\infty i_j\left(\frac{x}{1+x}\right)^j = \sum_{j=0}^\infty i_j x^j\frac{1}{(1+x)^j} =
  i_0 + \sum_{j=1}^\infty i_j x^j\left(\sum_{k=0}^\infty (-1)^k \binom{j+k-1}{k} x^k\right).\]
Thus, we have
  \[r_n = i_0 + \sum_{k=0}^{n-1} (-1)^k \binom{n-1}{k} i_{n-k}.\]

We now replace $i_{n-k}$ with the (first few terms of the) asymptotic formula for
$i_t$ from Theorem~\ref{thm:unlabelled}, and obtain
\begin{eqnarray*}
r_n &\sim& i_0 + \sum_{k=0}^{n-1} (-1)^k \binom{n-1}{k} (n-k)! \sqrt{n-k} \left( \frac{6}{\pi^2}\right)^{n-k}
\left( C_0 + \frac{C_1}{n-k} + \cdots \right) \\
&=& i_0 + (n-1)! \left( \frac{6}{\pi^2}\right)^n
\sum_{k=0}^{n-1} \frac{(-1)^k}{k!} \left( \frac{\pi^2}{6} \right)^k (n-k)^{3/2} \left( C_0 + \frac{C_1}{n-k} + \cdots \right) \\
&=& i_0 + n! \sqrt n \left( \frac{6}{\pi^2}\right)^n
\sum_{k=0}^{n-1} \frac{(-1)^k}{k!} \left( \frac{\pi^2}{6} \right)^k \left( 1 - \frac{k}{n} \right)^{3/2}
\left( C_0 + \frac{C_1}{n-k} + \cdots \right).
\end{eqnarray*}
We can now write the sum
$$
\sum_{k=0}^{n-1} \frac{(-1)^k}{k!} \left( \frac{\pi^2}{6} \right)^k \left( 1-\frac{k}{n} \right)^{3/2} \left( C_0 + \frac{C_1}{n-k} + \cdots \right)
$$
as an asymptotic expansion
$$
D_0 + \frac{D_1}{n} + \frac{D_2}{n^2} + \cdots.
$$
The $D_i$ may all be computed explicitly from the $C_i$.  For instance, the leading term is 
$$
D_0 = C_0 \sum_{k=0}^\infty \frac{(-1)^k}{k!} \left( \frac{\pi^2}{6} \right)^k = C_0 e^{-\pi^2/6},
$$
and we also have 
\begin{eqnarray*}
D_1 &=& \sum_{k=0}^\infty \frac{(-1)^k}{k!} \left( \frac{\pi^2}{6} \right)^k \left( C_1 - \frac32 C_0 k \right)  \\
&=& C_1 e^{-\pi^2/6} - \frac32 C_0 \sum_{k=1}^\infty \frac{(-1)^k}{(k-1)!} \left( \frac{\pi^2}{6} \right)^k \\
&=& C_1 e^{-\pi^2/6} + \frac{\pi^2}{4} C_0 \sum_{j=0}^\infty \frac{(-1)^j}{j!} \left( \frac{\pi^2}{6} \right)^j \\
&=& \left( C_1 + \frac{\pi^2}{4} C_0 \right) e^{-\pi^2/6}.
\end{eqnarray*}
Therefore
\[ r_n \sim n! \sqrt n \left( \frac{6}{\pi^2}\right)^n e^{-\pi^2/6} \left( C_0 + \frac{C_1 + \frac{1}{4} C_0 \pi^2}{n} + \cdots \right),\]
as desired.
%
\end{proof}

One interesting, and perhaps surprising, consequence of Theorem~\ref{thm:rigid-asympt} is that the proportion of unlabelled $n$-element interval
orders that are rigid tends to $e^{-\pi^2/6}\approx 0.193025$.  To understand this, it may help to note that $r_n/r_{n-1}$ and $i_n/i_{n-1}$
both behave as $6n / \pi^2$.  The number of unlabelled $n$-element interval orders in which there is exactly one pair of elements with duplicated
holdings is $r_{n-1}(n-1)$, which is asymptotically $i_n e^{-\pi^2/6} \pi^2/6$.  Similarly, for each fixed $k$, the number of unlabelled
$n$-element interval orders with exactly $k$ pairs of elements with duplicated holdings is equal to
$$
r_{n-k} \binom{n-k}{k} \simeq i_n \frac{\left( \frac{\pi^2}{6} \right)^k}{k!} e^{-\pi^2/6}.
$$
On the other hand, the number of unlabelled $n$-element interval orders in which there is some {\em triple} of elements with duplicated holdings
is at most $i_{n-2}(n-2) = O(i_n/n)$.
In other words, in a uniformly random unlabelled interval order, the number of pairs of elements with duplicated holdings is asymptotically
a Poisson random variable with mean $\pi^2/6$, while the probability that there is some triple with duplicated holdings tends to zero as
$n \to \infty$.


\section{Counting Labelled Interval Orders}

The best bounds on the number of labelled interval orders known to the authors appear in a paper by Brightwell, Grable, and Pr\"omel \cite{brightwell:forb-induced-posets}. In that paper, the authors use straightforward enumerative techniques to establish the following theorem.

\begin{thm}[Brightwell et al. \cite{brightwell:forb-induced-posets}]\label{thm:bound-labelled}
  Let $\ell_n$ denote the number of labelled interval orders on $n$ points. The numbers $\ell_n$ satisfy the following inequalities
  \[n^{2n-O\left(\frac{n\log\log n}{\log n}\right)}\leq \ell_n\leq \frac{(2n)!}{2^n}\leq n^{2n}\left(\frac{2}{e^2}\right)^n.\]
\end{thm}

Recall that the Stirling number of the second kind, denoted $S(n,k)$, is the number of ways to partition the set $[n]:=\{1,2,\dots,n\}$ into
$k$ non-empty parts. The number of surjections from $[n]$ to $[k]$ is $k! S(n,k)$.  There is a one-to-one correspondence between the set
of labelled interval orders on $[n]$ and the set of pairs $(R,S)$, where $R$ is an unlabelled rigid interval order on $k \le n$ points, and
$S$ is a surjection from $[n]$ onto the ground set of $R$.  We thus have
\[\ell_n = \sum_{k=1}^n r_k k! S(n,k).\]

Before using this sum to determine an asymptotic formula for $\ell_n$, we note the following estimate for the Stirling numbers of the second kind,
due to Hsu~\cite{hsu:stirling}.

\begin{lem}\label{lem:stirling-approx}
For each fixed $j \ge 0$,
\[S(n,n-j) \sim \frac{(n-j)^{2j}}{j! 2^j} \left( 1 + \frac{f_1(j)}{n-j} + \frac{f_2(j)}{(n-j)^2} + \cdots \right),\]
where the $f_i$ are explicit polynomials and $f_1(j) = j(2j+1)/3$.
\end{lem}

To get some intuition for the leading term in the lemma above, notice that the number of ways of partitioning $[n]$ into $n-j$ non-empty parts, 
with no part of size greater than~2, is exactly
$\displaystyle \frac{n(n-1) \cdots (n-2j+1)}{j! 2^j} \sim \frac{(n-j)^{2j}}{j! 2^j}$.  The number of ways of partitioning $[n]$
into $n-j$ non-empty parts, including a part of size~3, is of order $S(n,n-j) / n$.

We are now ready to establish an asymptotic formula for the number of labelled interval orders on $n$ points.

\begin{thm}\label{thm:labelled-asympt}
The number of labelled interval orders on $[n]$ is
  \[ \ell_n \sim (n!)^2 \sqrt{n} \left(\frac{6}{\pi^2}\right)^n \left( E_0 + \frac{E_1}{n} + \frac{E_2}{n^2} + \cdots \right), \]
where
  \[ E_0 = \frac{12\sqrt{3}}{\pi^{5/2}} \]
and the other $E_i$ are explicitly computable.
\end{thm}


\begin{proof}
We have
\begin{eqnarray*}
\ell_n &=& \sum_{j=0}^{n-1} r_{n-j} (n-j)! S(n,n-j) \\
&\sim& \sum_{j=0}^{n-1} \left(D_0 + \frac{D_1}{n-j} + \cdots \right) (n-j)!\sqrt{n-j}\left(\frac{6}{\pi^2}\right)^{n-j} (n-j)! S(n,n-j),
\end{eqnarray*}
where $D_0 = 12\sqrt{3}/(\pi^{5/2}e^{\pi^2/12})$.

Applying Lemma~\ref{lem:stirling-approx}, we now have
\begin{eqnarray*}
\ell_n &\sim& \sum_{j=0}^{n-1} \left(D_0 + \frac{D_1}{n-j} + \cdots \right) (n-j)!^2\sqrt{n-j}
\left(\frac{6}{\pi^2}\right)^{n-j}\cdot\\
&&\mbox{}\qquad\qquad\frac{(n-j)^{2j}}{2^j j!} \left( 1 + \frac{f_1(j)}{n-j} + \cdots \right) \\
&=& n!^2 \sqrt{n} \left(\frac{6}{\pi^2}\right)^n\,\, \sum_{j=0}^{n-1} \frac{1}{j!} \left( \frac{\pi^2}{12} \right)^j \left( \frac{(n-j)!(n-j)^j}{n!}\right)^2\\
&&\mbox{} \qquad\qquad
\sqrt{1-j/n} \left(D_0 + \frac{D_1}{n-j} + \cdots \right)\left( 1 + \frac{f_1(j)}{n-j} + \cdots \right) \\
&\sim& n!^2 \sqrt{n} \left(\frac{6}{\pi^2}\right)^n \left(E_0 + \frac{E_1}{n} + \cdots \right),
\end{eqnarray*}
where $E_0 = e^{\pi^2/12} D_0 = 12 \sqrt 3 \pi^{-5/2}$, as claimed.
\end{proof}

We may also find the distribution of the number of pairs of elements with duplicated holdings in a uniformly random labelled $n$-element
interval order.  The number of interval orders on $[n]$ with exactly $k$ such pairs is
$$
r_{n-j} (n-j)! S(n,n-j) \simeq e^{-\pi^2/12} \frac{ \left( \frac{\pi^2}{12} \right)^j }{j!} \ell_n.
$$
As before, the proportion of labelled $n$-element interval orders with a triple of elements with duplicated holdings tends to~zero.
In other words, the number of pairs of elements with duplicated holdings in a uniformly random {\em labelled} $n$-element interval order
is asymptotically a Poisson random variable with mean $\pi^2/12$.

Results about the number of pairs of elements with duplicated holdings (or, equivalently, about the automorphism group) come for free from
the methodology.  By contrast, results about the height or width of a uniformly random interval order, labelled or unlabelled, are likely to
be much harder to come by.

\bibliography{main}
\bibliographystyle{acm}

\end{document}